\documentclass{article}
\usepackage{amsmath,amssymb,amsthm,hhline} 

\ifx\pdfoutput\undefined
\usepackage{graphicx}
\else
\usepackage[pdftex]{graphicx}
\fi

\newenvironment{Proof}{{\it Proof.}}{{$\square$} \vskip.5cm} \newtheorem{Thm}{Theorem}[section]
\newtheorem{Lem}[Thm]{Lemma} \newtheorem{Pro}[Thm]{Proposition} \newtheorem{Cor}[Thm]{Corollary}

\theoremstyle{definition}

\theoremstyle{remark}

\newcommand{\R}{\mathbb{R}} \newcommand{\Z}{\mathbb{Z}}  
\newcommand{\Hyp}{\mathbb{H}}  \newcommand{\Ends}{{\rm{Ends}}}
\newcommand{\Stab}{{\rm{Stab}}}

\title{Poisson--Furstenberg boundaries \\ of fundamental groups of closed
3-manifolds}
\author{A. Malyutin\footnote{The first author is grateful to V.~Kaimanovich for useful discussions. The research is supported by RFBR grants 14-01-00373-a and 13-01-12422-ofi-m.}, P. Svetlov\footnote{Supported by RFBR grant 14-01-00062.}}

\begin{document}
\maketitle
\begin{abstract}
We obtain a description of Poisson--Furstenberg boundaries for (random walks on)
fundamental groups of compact graph-manifolds.
Together with previously known results due to V.A.~Kaimanovich and others, this
allows one to obtain descriptions of Poisson--Furstenberg boundaries for fundamental
groups of all closed 3-manifolds.
\end{abstract}

\section{Introduction}

This paper concerns the Poisson--Fur\-sten\-berg boundaries of (random walks on)
fundamental groups of 3-manifolds. If a closed (i.e., compact without boundary)
3-manifold is not a graph-manifold, then its fundamental group belongs to those
classes of groups whose Poisson--Furstenberg boundaries, in the case of reasonable measures, have already been described
in literature (see Sec.~3 for details).

Below, we present a description of Poisson--Furstenberg boundaries for the case of graph-manifolds.
Specifically, we show that the Poisson--Fur\-sten\-berg boundary of the fundamental group $\pi_1(M)$ of a compact graph-manifold $M$ can be identified, in the case of measure with finite entropy and finite first logarithmic moment, with the space of ends of Bass--Serre covering tree for the graph of groups associated with~$M$ via its geometric decomposition.

The key ingredient allowing us to obtain this result is the following theorem due to
the first author.

\begin{Thm}\label{Th0}
Let a countable group $G$ act on an infinite (simplicial) tree~$\Gamma$. Let $\mu$ be a probability measure on~$G$ whose support generates $G$ as a semigroup. Assume that $\Gamma$ has no proper $G$-invariant subtrees while the induced action of~$G$ on the set of ends $\Ends(\Gamma)$ has no finite orbits. 
Then the $G$-space $\Ends(\Gamma)$ has a unique $\mu$-stationary Borel probability measure $\nu_\mu$ and the pair $(\Ends(\Gamma),\nu_\mu)$ is a $\mu$-boundary of~$(G,\mu)$.

Assume moreover that $G$ is finitely generated, $\mu$ has finite entropy $-\sum_{G}\mu(g)\log\mu(g)$ and finite first logarithmic moment $\sum_G\log|g|\mu(g)$ with respect to a finite word norm $|\cdot|$ on~$G$. Assume also that there exist two vertices $u$, $v$ in $\Gamma$ such that the subgroup 
$$\Stab(u)\cap\Stab(v)$$
of their common stabilizers in $G$ is finite. Then the $\mu$-boundary $(\Ends(\Gamma),\nu_\mu)$ is isomorphic to the Poisson--Furstenberg boundary of $(G,\mu)$.
\end{Thm}

The proof of this theorem (in a stronger form) will be presented in~\cite{Mal14}.
The proof of the first assertion of the theorem is based on the approach of H.~Furstenberg~\cite{Fur63, Fur71}.
This approach essentially uses compactness of the $G$-space, while the space of ends $\Ends(\Gamma)$ of not locally finite tree may be non-compact. 
In order to apply the approach of Furstenberg to $\Ends(\Gamma)$, we consider a compact topology on $\Gamma\cup\Ends(\Gamma)$. This topology was introduced, under various names for various tree-like spaces, in
\cite{Wa75,Wa80,Ni89,Bo99,Sw99,MSh04,FJ04,CHL07} (see~\cite{Mal12} for details).
It can also be described in terms of horofunctions~\cite{Gr81} or as a natural quotient of the Cartwright--Soardi--Woess compactification~\cite{CSW93}.
For the particular case of locally finite trees (and compact $\Ends(\Gamma)$), the first assertion of Theorem~\ref{Th0} was established (via the same methods of~\cite{Fur63, Fur71}) by D.I.~Cartwright and P.M.~Soardi~\cite{CS89}. The related fact of convergence to infinite words for random walks on free groups with non-finitary measures was mentioned earlier in~\cite[Sec.~6.8]{KV83} without proof attributed to G.\,A.~Margulis.
The proof of the second assertion of Theorem~\ref{Th0} uses the ``strip criterion'' due to V.A.~Kaimanovich~\cite{Ka00-hyp}.

We will not give definitions related to random walk and boundary theory. This definitions can be found, e.g., in~\cite{KV83,Ka00-hyp}.

In Section~2 below, we give definitions related to graph-manifolds and study a natural action of the fundamental group of a graph-manifold on the corresponding Bass--Serre covering tree.  We show that this action satisfies the prerequisites of Theorem~\ref{Th0}, which yields  a description of Poisson--Fur\-sten\-berg boundaries for fundamental groups of all compact graph-manifolds.

In Section~3, we list results that provide descriptions of Poisson--Fur\-sten\-berg boundaries for fundamental groups of closed 3-manifolds that are not graph-manifolds.

\section{The case of graph-manifolds}  
We briefly recall basic definitions concerning graph-manifolds (see \cite{AFW,BS}  for more details and related references).

A (closed connected irreducible) 3-manifold $M$ is called a \emph{graph-manifold} if $M$ is not a Seifert fibered space and there exists a finite collection $\mathcal{T}\subset M$  of disjointly embedded incompressible tori such that each connected component of $M\setminus\mathcal{T}$ is a Seifert fibered space. Any such collection of tori with minimal number of components is unique up to isotopy and said to be a \emph{JSJ-surface}. 

Let $\mathfrak{M}_0$ be the class  formed by all orientable graph-manifolds  with the following property: 
for each $M\in\mathfrak{M}_0$ with JSJ-surface $\mathcal{T}\subset M$ all connected components of the complement $M\setminus\mathcal{T}$ are interiors of  trivial $S^1$-fibrations over orientable surfaces  of
negative Euler characteristic.

Let $\mathfrak{M}_s$ be the class formed by all graph-manifolds  of the class $\mathfrak{M}_0$ with the following additional property:
each torus $T\subset\mathcal{T}$  belongs to  closures of two different connected components of $M\setminus\mathcal{T}$. 

Note that each graph-manifold admits a finite-sheeted
covering by a manifold of the class $\mathfrak{M}_0$ \cite[Prop.~14.8]{BK96}, while each manifold of the class $\mathfrak{M}_0$  has a 2-folded covering manifold of the class $\mathfrak{M}_s$.

\emph{Blocks} of a manifold $M\in\mathfrak{M}_s$ are closures in $M$ of connected components of $M\setminus\mathcal{T}$. So each block is a submanifold of $M$ with toric boundary.

\paragraph{Graphs.} Let $M$ be a manifold of the class $\mathfrak{M}_s$, $\mathcal{T}$ be its JSJ-surface, $E$ be a set indexing the tori of $\mathcal{T}=\bigcup_{e\in E} T_e$, and $V$ be a set indexing the blocks of the
JSJ-decomposition: $M=\bigcup_{v\in V}M_v$. Consider the graph $\Gamma_M(V,E)$ with vertices and edges also indexed by the sets $V$ and $E$ such that  two vertices $v$ and $v'$ of $\Gamma_M$ are connected by an edge $e\in E$
if $T_e\subset M_v\cap M_{v'}$ (multiple edges but no loops are possible). Note that each finite-sheeted covering space of a manifold 
$M\in\mathfrak{M}_s$ also belongs to $\mathfrak{M}_s$.

Any covering space $N$ of $M$ also has a graph-structure. 
Let $p:N\to M$ be the covering map. \emph{Blocks} of the manifold $N$ are closures in $N$ of connected components of $N\setminus p^{-1}(\mathcal{T})$. So each block is a submanifold of $N$.

Similarly to the previous case we fix a set $U$ indexing the blocks of~$N$:
$$
N=\bigcup_{v\in V} p^{-1}(M_v)=\bigcup_{u\in U}N_u,
$$ 
and a set $F$ indexing the connected components of $p^{-1}(\mathcal{T})$:
$$p^{-1}(\mathcal{T})=\bigcup_{f\in F}P_f.$$ 
The graph $\Gamma_N(U,F)$ is constructed exactly as above:
two vertices $u,u'\in U$ are connected by the edge 
$f\in F$ if  $P_f\subset N_u\cap N_{u'}$.

We have a natural map $$p_*\colon\Gamma_N(U,F)\to\Gamma_M(V,E)$$ 
defined by $$M_{p_*(u)}=p\,(N_u),\qquad T_{p_*(f)}=p\,(P_f).$$

\paragraph{Action.} It is well known that the fundamental group of a manifold $M$ is isomorphic to the deck transformation\footnote{A homeomorphism $g\colon N\to N$ is a deck transformation of a covering $p\colon N\to M$ if $p\circ g=p$.} group $D_N$
of its universal covering space~$N$. 
If $M$ is a graph-manifold of the class $\mathfrak{M}_s$ and $p\colon N\to M$ is the universal covering map, then any element $g\in D_N$ permutes the blocks of $p^{-1}(M_v)\subset N$ for each $v\in V$ and we have a natural action of $D_N$ on~$\Gamma_N$.
 
Note that each element $g\in D_N$ defines  a class $L_g$ of freely homotopic loops 
 in $M$: a loop $\gamma$ belongs to $L_g$ if there exists a path  
$\tau\colon I\to N$ {such that} $\tau(1)=g(\tau(0))$ and $\gamma=p\circ\tau$. There exists
a natural one-to-one correspondence between the classes of conjugate elements in $D_N$
and the classes of freely homotopic loops in $M$ (if $\gamma\in L_g$, then we say that  $g$ and $\gamma$ 
and corresponding). 
 
We will repeatedly use the following lemma.

\begin{Lem}\label{s}  
Let $K$ be a connected $\pi_1$-injectively embedded submanifold of a manifold $M$.

\begin{itemize}
 \item[1.] 
If $p\colon\widetilde{M}\to M$ is the universal covering and $K'$ is a connected component of $p^{-1}(K)$, then the restriction $p\,|_{K'}\colon K'\to K$ is a universal
covering. 
Moreover, the restrictions onto $K'$ of transformations from the set $$D'=\{g\in D_{\widetilde{M}}\colon\,g(K')=K'\}$$
form the deck transformation group $D_{K'}$ of $p\,|_{K'}$.

\item[2.] For each element $g\in D_{\widetilde{M}}$ the following two conditions are equivalent.
\begin{itemize}
 \item[a)] There exists a connected component
 $K'\subset p^{-1}(K)$ \\such that {${g(K')=K'}$}.
 \item[b)] There exists a loop $\gamma\in L_g$ contained in $K$.
\end{itemize}
\end{itemize}
\end{Lem} 

\begin{Proof} 1. We have the following commutative diagram
$$
\begin{array}{ccccc}
&K'&\stackrel{i}{\longrightarrow} &\widetilde{M}&\\
{{p|_{K'}}}&\downarrow&&\downarrow&{p,}\\
&K&\stackrel{j}{\longrightarrow}&M&
\end{array}
$$
where $i$ and $j$ are inclusions. The induced map 
$$(j\circ p|_{K'})_*\colon\pi_1(K')\to \pi_1(M)$$ 
is injective. 
The other composition 
$$(p\circ i)_*\colon \pi_1(K')\to\pi_1(M)$$ 
is the trivial homomorphism due to the simply-connectedness of $\widetilde{M}$. 
Therefore $\pi_1(K')$ is trivial and $K'$ is simply-connected. Hence $p\,|_{K'}\colon K' \to K$ is a universal covering.

If $g\in D'$ then $p\,|_{K'}\circ g\,|_{K'}=p\,|_{K'}$, therefore $D'$ is a subgroup of $D_{K'}$. The equality
immediantly follows from the $\pi_1$-injectivity.

2. [a)$\Rightarrow$b)] Let $x\in K'$ be a point.  Since $K'$ is path-connected, there is a curve $\tau\colon I\to K'$  such that $\tau(0)=x$, $\tau(1)=g(x)$.
The loop $p\circ\tau\in L_g$ is contained in $K$.

 [b)$\Rightarrow$a)] Let $\tau\colon I\to \widetilde{M}$ be a path such that $g(\tau(0))=\tau(1)$ and the image of $p\circ\tau$
is contained in $K$. The points $\tau(0),\tau(1)$ belong to one and the same connected component $K'$ of $p^{-1}(K)$.
Each deck transformation map $g$ permutes the connected components of $p^{-1}(K)$. Therefore $g(K')=K'$.
\end{Proof}

\begin{Lem} \label{CorLem}
Let $M$ be a graph-manifold of the class $\mathfrak{M}_s$,  $p\colon N\to M$ be the universal covering, and $\Gamma_N(U,F)$ be the graph of~$N$. 
\begin{enumerate}
\item The restriction maps $N_u\to M_{p_*(u)}$ and $P_f\to T_{p_*(f)}$ are  universal coverings  for all $u\in U$, $f\in F$.
\item $\Gamma_N$ is a regular tree of countably infinite degree.
\end{enumerate}
\end{Lem}
\begin{Proof} 1.  The incompressible tori of JSJ-decomposition are  $\pi_1$-injective because $M$ is orientable.
So, the assertion follows from item~1 of Lemma~\ref{s}.

2. It is explaned in \cite[\S 1.5]{BK00} that $\Gamma_N$ is isomorphic to the Bass--Serre covering tree of the graph of groups 
$$(\Gamma_M(V,E),\pi_1(M_v),\pi_1(T_e)).$$ 
Since each edge group $\pi_1(T_e)$ has infinite index in $\pi_1(M_v)$ (whenever $T_e\subset \partial M_v$),
it follows that each vertex of $\Gamma_N$ has (countably) infinite degree. 
\end{Proof}
 
Now, we are looking for stabilizers
$$\Stab(u)=\{g\in D_N \colon\,g(u)=u\}$$ of vertices of $\Gamma_N$. By Lemma~\ref{s}, we can (and  will) identify $\Stab(u)$ and $D_{N_u}$ -- 
the deck transformation group of the universal covering $N_u\to M_{p_*(u)}$, and also $\Stab(f)$ and $D_{P_f}$ -- 
the deck transformation group of the universal covering $P_f\to T_{p_*(f)}$, for all $u\in U$, $f\in F$.

\begin{Lem}\label{stabs}
Let $M$ be a graph-manifold of the class $\mathfrak{M}_s$, and $\Gamma_N(U,F)$ be the graph of its universal covering space.
\begin{itemize}
%
\item[1)] If $u,v\in U$ are two vertices with $|uv|=2$, then $\Stab(u)\cap\Stab(v)$ is an infinite cyclic subgroup 
whose elements correspond to the powers of the fiber class of the block $M_{p_*(w)}$, where $w\in U$ is the vertex between $u$ and~$v$.

\item[2)] For each pair of vertices $u,w\in U$  with $|uw|\ge 3$, the intersection subgroup
 $\Stab(u)\cap \Stab(w)$ is trivial.
\end{itemize}
\end{Lem}

\begin{Proof}
%
%
 %
%
%
{\it 1)}  Let $u,v\in U$ be vertices with $|uv|=2$ and $w\in U$ be the vertex between them. 
If $g(u)=u$ and $g(v)=v$, then $g(w)=w$ because $\Gamma_N$ is a tree. 
Consequently, we have $g(N_w)=N_w$,
$g(P_f)=P_f$, and $g(P_{f'})=P_{f'}$, where  $P_f=N_u\cap N_w$ and $P_{f'}=N_w\cap N_v$. 

Recall that $p|_{N_w}\colon N_w\to M_{p_*(w)}$ is a universal covering map (see Lemma~\ref{CorLem}) and the restriction $g_w$ of $g$
on $N_w$ is in $D_{N_w}$ (Lemma~\ref{s}). Applying  Lemma~\ref{s} to the pairs $T_{p_*(f)}\subset M_{p_*(w)}$ and 
$T_{p_*(f')}\subset M_{p_*(w)}$, we see that there exist loops $\gamma\subset T_{p_*(f)}$ and $\gamma'\subset T_{p_*(f')}$ of the class $L_{g_w}$, and that they are freely homotopic in $M_{p_*(w)}$. Let $F:S^1\times I\to M_{p_*(w)}$ be a free homotopy between them.

Since $M$ belongs to the class $\mathfrak{M}_s$, the block $M_{p_*(w)}$ is homeomorphic to a trivial Seifert fibration with base surface $B$ (of negative Euler characteristic). 
Let $\phi\colon M_{p_*{(w)}}\to B$ be the fibering map.
Then the composition $\phi\circ F$ is a homotopy between the loops $\phi(\gamma)$ and $\phi(\gamma')$ that lie on distinct boundary components of~$B$. 
It follows that the loops $\phi(\gamma)$ and $\phi(\gamma')$ are contractible which implies that the loops $\gamma$ and $\gamma'$ are homotopic to some powers of a fiber of $\phi$. 

In other words, we have shown that each element of $\Stab(v)\cap\Stab(u)$ is in $\Stab(w)$ and corresponds to some power of the fiber class of the block $M_{p_*(w)}$.
Next, we show that those elements in $\Stab(w)\simeq D_{N_w}$ that correspond to the powers of the fiber class of $M_{p_*(w)}$ all lie in $\Stab(v)\cap\Stab(u)$ and form an infinite cyclic subgroup.

Indeed, observe that the universal covering space $N_w$ of the block $M_{p_*(w)}=S^1\times B$ has the structure of the product $R\times \widetilde{B}$, where $R$ and $\widetilde{B}$ are the universal covering spaces for $S^1$ and $B$, respectively. We also have
$$
D_{N_w}\simeq\pi_1(M_{p_*(w)})=\pi_1(S^1)\times\pi_1(B),
$$
where $\pi_1(S^1)$ is an infinite cyclic and $\pi_1(B)$ is a non-Abelian free group.
Therefore, the powers of the fiber class form an infinite cyclic subgroup in~$\pi_1(M_{p_*(w)})$, which is the center of~$\pi_1(M_{p_*(w)})$. Consequently, in $D_{N_w}$, the elements corresponding to the powers of the fiber class (this correspondence is defined up to the conjugacy equivalence) form an infinite cyclic subgroup~$Z\subset D_{N_w}$. Clearly, all elements of $Z$ preserve the planes of $\partial N_w$ whence it follows that $Z\subset \Stab(v)\cap\Stab(u)$.

We have thus proved that $Z=\Stab(v)\cap\Stab(u)$.

{\it 2)} Let $u,w\in U$ be vertices with $|uw|\ge 3$. Let $g\in \Stab(u)\cap\Stab(w)$. Since $\Gamma_N$ is a tree, there is a unique path between $u$ and $w$ in $\Gamma_N$ and $g$ fixes all vertices of this path.
In particular, since $|uw|\ge 3$, there exist vertices $u_1,u_2,u_3,u_4\in\Gamma_N$ with $|u_iu_j|=|i-j|$ such that $g(u_i)=u_i$ for $i=1,2,3,4$.
Let $f_i\in F$ be the edge connecting $u_i$ and $u_{i+1}$ ($i=1,2,3$). 

Recall that the blocks $M_{p_*(u_i)}$ are homeomorphic to  trivial Seifert fibrations with some base surfaces~$B_i$.
Let $\phi_i\colon M_{p_*{(u_i)}}\to B_i$ be the fibering maps.
Applying the argument from the above proof of item 1 to the pair $u_1,u_3\in U$, we see that $g$ acts on the plane $P_{f_2}$ as a shift along the lines
\begin{equation}\label{f1}
\{(\phi_2\circ p)^{-1}(x) \mid x\in \phi_2(T_{p_*(f_2)})\}
\end{equation}
(which are the inverse images of fibers of~$\phi_2$).
The same argument applied to the pair $u_2,u_4\in U$ shows that $g$ acts on $P_{f_2}$ as a shift along lines
\begin{equation}\label{f2}
\{(\phi_3\circ p)^{-1}(x) \mid x\in \phi_3(T_{p_*(f_2)})\}.
\end{equation}
The foliations \eqref{f1} and \eqref{f2} of $P_{f_2}$  are transversal by the minimality of the JSJ-surface.
Therefore, $g$ acts on $P_{f_2}$ identically. 
\end{Proof}

\begin{Lem} Let $M$ be a graph-manifold of the class $\mathfrak{M}_s$, and $\Gamma_N(U,F)$ be the graph of its universal covering space.

Then for any vertex $u\in U$ of the graph $\Gamma_N$ there exists an element $g\in \Stab(u)$ such that
$u$ is the only vertex of $\Gamma_N$ fixed by $g$.
\end{Lem}

\begin{Proof} Let $u\in U$. 
Consider a trivial Seifert fibration $\phi\colon M_{p_*(u)}\to B$. Take a loop $\gamma\colon I\to M_{p_*(w)}$ such that the loop $\phi\circ\gamma$ is non-peripheral (not freely homotopic to a loop lying in $\partial B$). Then $\gamma$ is non-peripheral in $M_{p_*(u)}$.

Let $g_\gamma$ be an element corresponding to~$\gamma$ and lying in $\Stab(u)$.
Then $g_\gamma$ preserves no plane from $\partial N_u$ (because otherwise $\gamma$ would be peripheral by assertion~2 of Lemma~\ref{s}). 
This clearly means that $g_\gamma$ permutes the connected components of $\Gamma_N\setminus\{u\}$ fixing no one of them.
\end{Proof}

\begin{Cor}\label{orbsubtr}Let $M$ be a graph-manifold of the class $\mathfrak{M}_s$, and $\Gamma_N(U,F)$ be the graph of its universal covering space.
Then the action of $D_N$ on  $\Gamma_N\cup \Ends (\Gamma_N)$ has neither finite orbits nor proper invariant subtrees.
\end{Cor}

\begin{Proof} Assume to the contrary that there exists a finite orbit or a proper invariant subtree
$$
A\subset \Gamma_N \cup \Ends (\Gamma_N).
$$ 
Since $\Gamma_N$ has an infinite number of vertices of infinite degree (see Lemma~\ref{CorLem}), we can take a vertex $u\in U$
such that $A$ lies in a connected component of 
$$
\Gamma_N \cup \Ends (\Gamma_N) \setminus\{u\}.
$$ 
As it follows from the previous lemma, there exists an element $g\in \Stab(u)$ which permutes connected components of $\Gamma_N\setminus\{u\}$ (fixing no one of them). This means that $A$ is not $G$-invariant which contradicts to our assumption. 
\end{Proof}

Now we recall the {\it geometric decomposition theorem} for 3-manifolds\footnote{
A manifold is called {\it geometric} if its interior admits a complete 
locally homogeneous Riemannian metric of finite volume.}.
\begin{Lem}\label{geomgm}(Thurston; see, e.g.,  \cite[Corollary 1.45]{K}.) 
Let M be a Haken (3-dimensional compact $\mathbb{P}^2$-irreducible 
containing a properly embedded two-sided incompressible surface) manifold.
Then $M$ admits a decomposition
along a finite collection $\mathcal{G}_M$ of disjoint
incompressible tori and Klein bottles into geometric components.
\end{Lem}
Note that any graph-manifold is $\mathbb{P}^2$-irreducible (i.\,e.,~does not contain a two-sided projective plane) and fibered  tori
in each Seifert fibered component are incompressible,
so all graph-manifolds are Haken.
If $M$ is a graph-manifold, then the corresponding collection $\mathcal{G}_M$ of tori and Klein bottles is nonempty and 
the corresponding geometric components are Seifert fibrations with boundary. Therefore\footnote{Each geometric manifold with boundary has geometry modelled either on $\Hyp^3\times\R$ or on $\Hyp^3$  (see \cite[\S~1.6]{AFW}).} they admit the geometric structure modelled on $\mathbb{H}^2\times\mathbb{R}$. Note that for graph-manifolds of the class $\mathfrak{M}_0$ (and $\mathfrak{M}_s$) the geometric decomposition coincides with the JSJ-one.

Let $M$ be a graph-manifold, $p_s\colon M_s\to M$ its finite-sheeted covering with
$M_s\in\mathfrak{M}_s$, $\Gamma_N(U,F)$ be the graph of their
common universal covering space. It is clear that $p_s^{-1}(\mathcal{G}_M)$
is a JSJ-surface of ${M_s}$.
Therefore each element of the deck transformation group of the covering $N\to M$ permutes the blocks of $N$.
This allows us to define an action of $\pi_1(M)$ on the tree $\Gamma_N$.

\begin{Pro} Let $M$ be a graph-manifold, 
and $\Gamma_N(U,F)$ be the graph of its universal covering space.  
Then the action of $\pi_1(M)$ on $\Gamma_N\cup \Ends (\Gamma_N)$ satisfies the conditions of Theorem \ref{Th0}.
\end{Pro}
\begin{Proof} First note that if a group $G$ acts on a tree $T$ such that the conditions of Theorem~\ref{Th0} are fulfilled
for some subgroup $H\subset G$ of finite index, then these conditions are clearly fulfilled for $G$ also.

 Let $p_s\colon M_s\to M$ be a finite-sheeted covering and $M_s\in\mathfrak{M}_s$.
As it follows from Corollary \ref{orbsubtr} and  item 2 of Lemma \ref{stabs}, the action of (finitely generated group) $\pi_1(M_s)$
on the tree $\Gamma_N$ satisfies the conditions of Theorem \ref{Th0} (recall that the fundamental group of any compact manifold is finitely generated). 

Let us fix a point $x\in M_s$. 
The group $\pi_1(M_s,x)$ is a subgroup of a finite index in $\pi_1(M,p_s(x))$. The action of $\pi_1(M,p_s(x))$ on the tree $\Gamma_N$
extends the action of $\pi_1(M_s,x)$. So the action of  $\pi_1(M)$ on $\Gamma_N\cup \Ends (\Gamma_N)$ satisfies the  conditions of Theorem \ref{Th0} also.
\end{Proof}

\section{The remaining cases} 
In this section we list results that provide descriptions of Poisson--Fur\-sten\-berg boundaries for fundamental groups of closed 3-manifolds that are not graph-manifolds.

Recall that, for certain types of measures, the problem of describing the Poisson--Fur\-sten\-berg boundary of a group reduces to the problem of describing the boundary for a subgroup of finite index.

\begin{Lem}
\label{virtual}  
\cite{Ka91-solv, Fur71}
Let $G$ be a discrete group, $\mu$ a probability measure on~$G$, and $F$ a subgroup of finite index in~$G$.  Define a probability measure $\nu$ on $F$ as the distribution of the point where the random walk 
issued from the identity of $G$ returns for the first time to~$F$.
Then the Poisson--Furstenberg boundaries of $(G,\mu)$ and $(F,\nu)$ are isomorphic.

\cite{Ka91-solv,Fo14,HLT12} If, in addition, $G$ is finitely generated and the measure $\mu$ has finite first moment (resp., finite entropy, finite first logarithmic moment) in~$G$, then $\nu$ has finite first moment (resp., finite entropy, finite first logarithmic moment) in~$F$. 
\end{Lem}

\begin{proof}
The first assertion of the lemma follows from Lemma~2.2 of~\cite{Ka91-solv} (see also Sec.~4.3 in~\cite{Fur71}).

The assertion about the finiteness of entropy is well known to specialists. It follows, e.\,g., from Lemma~3.4 of~\cite{Fo14}. 
See also~\cite{HLT12}.

The assertion concerning the first moment is proved in~\cite[Lemma 2.3]{Ka91-solv} for the case of normal subgroup. The idea of the proof easily extends to the case of arbitrary subgroup of finite index: instead of the random walk on the finite factor-group, one should consider the Markov process on the set of cosets, with some additional minor modifications. The same argument works equally well for the case of first logarithmic moment.
(See also Theorem~4.5 in~\cite{Fo14}.)
\end{proof}

In all those cases where the fundamental group of a closed 3-manifold has non-trivial Poisson--Furstenberg boundary, the most general results obtained so far are due to V.~Kaimanovich and describe the boundary either for the case of measure with finite entropy and finite first logarithmic moment or for the less general case of measure with finite first moment. 
These types of measure are covered by Lemma~\ref{virtual}, and in the following analysis we will repeatedly use the transition to the subgroups of finite index.
In particular, by virtue of Lemma \ref{virtual}, we can restrict further exposition to the case of orientable manifolds.

Recall that a manifold is called {\it prime} if it can not be presented as a nontrivial connected sum. 
We have the following ad hoc classification lemma for orientable prime 3-manifolds.

\begin{Lem}\label{3-topology} If $M$ is a (closed connected) orientable prime 3-manifold, then $M$  is either 
\begin{enumerate} 
\item geometric, or 
\item admits a metric of non-positive sectional curvature, or 
\item is a graph-manifold. 
\end{enumerate}
 \end{Lem} 
Note that there are geometric manifolds and graph-manifolds that admit metrics of non-positive sectional curvature.

\begin{Proof} If $M$ is reducible (i.\,e., $M$ contains an
embedded 2-sphere that does not bound a ball) then it is homeomorphic to $S^2\times S^1$
and admits the geometric structure modelled on $S^2\times\mathbb{R}$. 

If $M$ is  non-Haken irreducible manifold, then it is geometric by the Geometrization conjecture proved by G.~Perelman.

If $M$ is a Haken irreducible manifold, then by Lemma \ref{geomgm}  $M$  is 
either geometric (empty collection $\mathcal{G}_M$) or
contains a non-empty collection  $\mathcal{G}_M$ of disjointly embedded incompressible tori and Klein bottles such that each
connected component of {${N\setminus\mathcal{G}_M}$} 
is a geometric manifold (with non-empty boundary).
So each component of $M\setminus \mathcal{G}_M$ is either 
Seifert fibered with geometry modelled on
$\mathbb{H}^2\times\mathbb{R}$ or hyperbolic \cite[\S~1.6]{AFW}.
 If $M\setminus\mathcal{G}_M$
contains only Seifert pieces, then it is a graph-manifold.
If $M\setminus\mathcal{G}_M$ contains a hyperbolic component,
then it admits a metric of non-positive sectional curvature \cite{L}. \end{Proof}

Lemma~\ref{3-topology} confirms the following classification of (closed connected) orientable 3-manifolds, see Figure \ref{diagram}.

\begin{figure}
\begin{center}
\includegraphics[viewport=350 -5 50 400,scale=0.55]{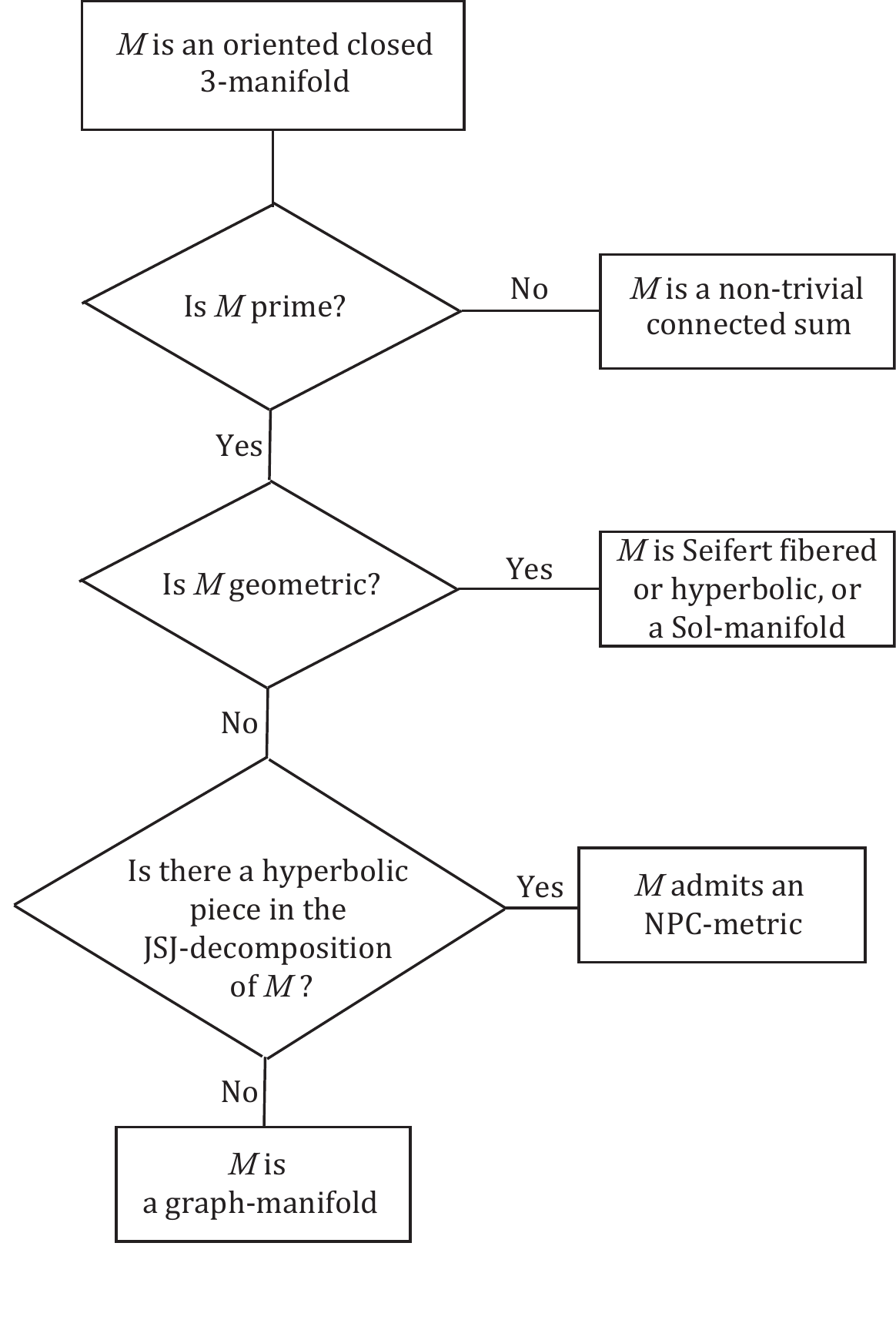}\\
\caption{Four types of 3-manifolds}\label{diagram}
\end{center}
\end{figure}

It can be mentioned that $\R P^3\# \R P^3$ is the only oriented non-prime geometric manifold, and that a geometric manifold which admits an NPC-metric has geometry modelled either on $\Hyp^2\times\R$ or on $\Hyp^3$. Note also that some graph-manifolds admit NPC-metrics~\cite{BK96}. There are no other intersections between terminal classes of the above diagram.

\subsection{The case of non-prime manifolds}
\label{sec:non-prime}
The fundamental group of a non-prime manifold is the free product of the fundamental groups of its prime summands. The 
fundamental groups of prime summands are non-trivial by the Poincare conjecture proved by G.~Perelman\footnote{Of~course, we can go without the conjecture: if some of the summands of a manifold $M$ were fake 3-spheres, 
one could remove them from $M$ to obtain a manifold $M'$ with the same fundamental group.}.
The free product of non-trivial finitely generated groups is either isomorphic to~$\Z/2\Z*\Z/2\Z$ or {\em has infinitely many ends} (see the Stallings theorem on ends of groups).
The group~$\Z/2\Z*\Z/2\Z$ is virtually Abelian.
Abelian groups have trivial Poisson--Furstenberg boundaries~\cite{Bl55}. It follows by Lemma~\ref{virtual} that virtually Abelian groups has trivial Poisson--Furstenberg boundaries. 
The Poisson--Furstenberg boundaries of finitely generated groups with infinitely many ends are described (for measures with finite entropy and finite first logarithmic moment)
in~\cite[Theorems~8.3,8.4]{Ka00-hyp}. See also earlier work \cite{Wo89}, where the case of finitely supported measure is considered. 
Note that the Stallings theorem on ends of groups implies that Poisson--Furstenberg boundaries of (finitely generated) groups with infinitely many ends can also be described (for measures with finite entropy and finite first logarithmic moment) via Theorem~\ref{Th0}.

\subsection{The cases of geometric manifolds}
Recall that a compact manifold is called {\it geometric} if its interior admits a complete locally homogeneous Riemannian metric of finite volume. 
A~geometric manifold $M$ is called an \emph{$X$-manifold} (or a \emph{manifold with geometry modelled on~$X$}) 
if there is a simply-connected Riemannian manifold~$X$ and a subgroup~$G$ of the isometry group $\operatorname{Isom}X$ such that $M=X/G$.
There are eight types of geometric manifolds in dimension~3. The six Seifert ones: 
$S^3$ (spherical), 
$\mathbb{E}^3$ (euclidean), 
$S^2\times\R$, 
$\Hyp^2\times\R$ (both with the product metric), 
$\widetilde{\mathrm{SL}_2\R}$ (with a left-invariant metric) and $\rm{Nil}$ (the Lie group of upper unitriangular $3\times 3$-matrices with a left-invariant metric). 
The two non-Seifert ones: 
$\mathrm{Sol}$ (the Lie group of isometries of 2-dimensional Minkowski space with a left-invariant metric) and $\Hyp^3$ (hyperbolic). 
For details see, e.g., the survey \cite{Sc}.

The only non-prime geometric manifold is $\R P^3\#\R P^3$ with fundamental group $\Z/2\Z*\Z/2\Z$; 
it admits geometry modelled on $S^2\times\R$. 

\subsubsection{The case of $S^3$-manifolds}\label{sec:finite}
Each $S^3$-manifold is finitely covered by $S^3$ (the Hopf $S^1$-fibration).
Therefore the fundamental group of every $S^3$-manifold is finite. 
Since the trivial group has trivial Poisson--Furstenberg boundary by definition, it follows by Lemma \ref{virtual} that each finite group
has trivial Poisson--Furstenberg boundary. 
 
Note that by the Elliptisation conjecture (which follows from G.~Perelman results) each 3-manifold with a finite fundamental group is Seifert fibered and admits the geometry modelled on $S^3$.

\subsubsection{The case of $\mathbb{E}^3$- and $S^2\times\R$-manifolds}\label{virtabel}
Each $\mathbb{E}^3$- or $S^2\times\R$-manifold is finitely covered by a trivial $S^1$-fibration ($T^3$ in the first case, and $S^2\times S^1$ in the second one) \cite[Table~1]{AFW}.
Therefore the fundamental groups of such manifolds are virtually Abelian: they contain Abelian groups as finite index subgroups ($\Z^3$~in the $\mathbb{E}^3$-case and $\Z$ in the $S^2\times\R$-case). Abelian groups have trivial Poisson--Furstenberg boundaries \cite{Bl55}. It follows by Lemma~\ref{virtual} that each virtually Abelian group has trivial Poisson--Furstenberg boundary. 

\subsubsection{The case of Nil-manifolds}\label{sec:nil}
Each Nil-manifold $M$ is finitely covered by an $S^1$-bundle over the torus $T^2$ (and does not covered by the 3-torus);
alternatively: $M$ is finitely covered by a torus-bundle $N\simeq T^2\times_{f^n} S^1$ over the circle (the mapping torus of a power of a Dehn twist $f\colon T^2\to T^2$) \cite[Table~1]{AFW}.
This implies that $\pi_1(N)$ has a presentation
$$
\langle a,b,c\,|\,ab=ba,\,cac^{-1}=a,\,cbc^{-1}=ba^n\rangle.
$$
Obviously the commutant of $\pi_1(N)$ is generated by $a^n$ and lies in the center of $\pi_1(N)$.
Therefore the fundamental group of every Nil-manifold is virtually nilpotent. 
Nilpotent groups have trivial Poisson--Furstenberg boundaries~\cite{DM},~\cite{Mar}.
It follows by Lemma~\ref{virtual} that each virtually nilpotent group has trivial Poisson--Furstenberg boundary.

\subsubsection{The case of Sol-manifolds}\label{sec:sol}
Each Sol-manifold $M$ is finitely covered (with degree $\le 2$) by a torus-bundle $N\simeq T^2\times_f S^1$ over the circle (the mapping torus of a pseudo-Anosov diffeomorphism $f\colon T^2\to T^2$) \cite[Table~1]{AFW}.
This implies that $\pi_1(N)$ has a presentation
$$
\langle a,b,c\,|\,ab=ba,\,cac^{-1}=a^kb^l,\,cbc^{-1}=a^mb^n\rangle,\quad kn-ml=1,\,|k+n|>2.
$$
Therefore the fundamental group of every Sol-manifold contains a subgroup of finite index $\le 2$ which is a polycyclic group $\Z^2\rtimes\Z$, where $\Z$ acts on $\Z^2$ by hyperbolic automorphisms.
The Poisson--Furstenberg boundaries of such polycyclic groups are described (for measures with finite first moment) in \cite[Theorem~1]{Ka95-pol}.

\subsubsection{The case of $\Hyp^2\times\R$-manifolds}\label{sec:h^2times r}
Each $\Hyp^2\times\R$-manifold $M$ is finitely covered by a manifold of the form $F\times S^1$, where $F$ is a compact hyperbolic surface \cite[Table 1]{AFW}.
Therefore the fundamental group $\pi_1(M)$ contains the product $\pi_1(F)\times \Z$ as a subgroup of finite index.
The product $\pi_1(F)\times \Z$ can be treated as a central extension of~$\pi_1(F)$. The case of such groups is discussed in the following section.

\subsubsection{The case of $\widetilde{\mathrm{SL}_2\R}$-manifolds}\label{sec:sl2r}
Each $\widetilde{\mathrm{SL}_2\R}$-manifold $M$ is finitely covered by the unit circle bundle\footnote{If $X$ is a Riemannian manifold then the unit sphere bundle $UX$ is defined as a subbundle of the tangent bundle~$TX$: $$UX=\{(x,v)\in TX: |v|=1\}.$$} $UF$ over a closed hyperbolic surface~$F$ \cite[Table~1]{AFW}.
In particular, $\pi_1(UF)$ is a $\Z$-extension of $\pi_1(F)$: we have an exact sequence
$$
0\to\Z\to\pi_1(UF)\to\pi_1(F)\to 1.
$$
It turns out that this extension is central (see, e.\,g., \cite[\S~4, Sec.~``The geometry of $\widetilde{\mathrm{SL}_2\R}$'']{Sc}).
Thus, the fundamental group $\pi_1(M)$ contains a central extension of $\pi_1(F)$ by $\Z$ as a subgroup of finite index. 
By \cite[Theorems~2.1.5,~5.2.7]{Ka95-cov}, the Poisson--Furstenberg boundary of a pair $(G,\mu_0)$, where $G$ is a central extension of a group~$H$, coincides with the Poisson--Furstenberg boundary of the pair $(H, \pi_*(\mu_0))$, where $\pi_*(\mu_0)$ is the image of $\mu_0$ under the projection $\pi\colon G\to H$ (we copy the argument from~\cite{FM}). 
Clearly, $\pi_*$ sends measures with finite first moment (resp., finite entropy, finite first logarithmic moment) to measures with finite first moment (resp., finite entropy, finite first logarithmic moment). 

Thus, by the preceding argument and Lemma~\ref{virtual}, the Poisson--Furstenberg boundary of $(\pi_1(M),\mu)$, where $\mu$ is a measure with finite entropy and finite first logarithmic moment, is isomorphic to the Poisson--Furstenberg boundary of $(\pi_1(F),\nu_\mu)$, where $\nu_\mu$ is determined by $\mu$ and has finite entropy and finite first logarithmic moment.

One can treat $\pi_1(F)$ as a hyperbolic group, as a discrete subgroup of the Lie group ${\rm SL}_2\mathbb{R}$, as an amalgam, or as an HNN-extension. Poisson--Furstenberg boundaries of hyperbolic groups for the case of measures with finite entropy and finite first logarithmic moment are described in~\cite[Sec.~7]{Ka00-hyp}. See also~\cite{An90}. 
For descriptions of Poisson--Furstenberg boundaries for discrete subgroups of semi-simple Lie groups see \cite[Sec.~10]{Ka00-hyp}, \cite{Ka85}, \cite{Fur71}, \cite{Gui80},~\cite{Le85}.
Finally, the group $\pi_1(F)$ can be presented as an amalgam or an HNN-extension such that the action on the corresponding Bass--Serre covering tree satisfies the prerequisites of Theorem~\ref{Th0}, which thus yields another description for the Poisson--Furstenberg boundaries of $\pi_1(F)$ and $\pi_1(M)$.

\subsubsection{The case of $\Hyp^3$-manifolds}\label{sec:hyp}
The fundamental groups of $\Hyp^3$-manifolds can be treated as hyperbolic groups or as discrete cocompact subgroups of the Lie group ${\rm Isom}\,\mathbb{H}^3\simeq {\rm PSL}_2\mathbb{C}$ (see Sec.~\ref{sec:sl2r} and also Sec.~\ref{sec:npc}).

Note that each $\Hyp^3$-manifold is finitely covered by an $F$-bundle $F\times_f S^1$ over the circle (the mapping torus of a pseudo-Anosov diffeomorphism $f\colon F\to F$, $F$ is a compact hyperbolic surface)\footnote{We do not need the assertion and mention it here for plenitude of exposition. The assertion was one of Thurston conjectures. It was proved in~\cite{AFW}.}.

\subsection{The case of non-positive sectional curvature}\label{sec:npc}
The universal covering space $N$ of a closed manifold $M$ with a metric of non-positive sectional curvature is a Cartan-Hadamard manifold. The fundamental group $\pi_1(M)$ viewed as the deck transformation group of $N$ is a cocompact group of isometries of~$N$. 
Poisson--Furstenberg boundaries of such groups are described in \cite{Ka00-hyp} (see Remark~2 in Sec.~9). See also \cite{KM}, \cite{BL}.

\end{document}